\documentclass{amsart}
\usepackage{amssymb}
\usepackage{amsmath}
\usepackage{amsfonts}

\newtheorem{theorem}{Theorem}
\theoremstyle{plain}

\newtheorem{corollary}{Corollary}

\newtheorem{proposition}{Proposition}

\numberwithin{equation}{section}
\input{tcilatex}

\begin{document}
\title{Notes concerning Codazzi pairs on almost anti-Hermitian manifolds}
\author{Aydin GEZER}
\address{Ataturk University, Faculty of Science, Department of Mathematics,
25240, Erzurum-Turkey.}
\email{agezer@atauni.edu.tr}
\author{Hasan CAKICIOGLU}
\address{Ataturk University, Faculty of Science, Department of Mathematics,
25240, Erzurum-Turkey.}
\email{h.cakicioglu@gmail.com}

\begin{abstract}
Let $\nabla $ be a linear connection on an $2n$-dimensional almost
anti-Hermitian manifold $M$\ equipped with an almost complex structure $J$,
a pseudo-Riemannian metric $g$ and the twin metric $G=g\circ J$. In this
paper, we first introduce three types of conjugate connections of linear
connections relative to $g$, $G$ and $J$. We obtain a simple relation among
curvature tensors of these conjugate connections. To clarify relations of
these conjugate connections, we prove a result stating that conjugations
along with an identity operation together act as a Klein group. Secondly, we
give some results exhibiting occurrences of Codazzi pairs which generalize
parallelism relative to $\nabla $. Under the assumption that $(\nabla ,J)$
being a Codazzi pair, \ we derive a necessary and sufficient condition the
almost anti-Hermitian manifold $(M,J,g,G)$ is an anti-K\"{a}hler relative to
a torsion-free linear connection $\nabla $. Finally, we investigate
statistical structures on $M$ under $\nabla $ ($\nabla $ is a $J-$invariant
torsion-free connection).
\end{abstract}

\maketitle

\section{\protect\bigskip \textbf{Introduction}}

A pseudo-Riemannian metric $g$ on a smooth $2n-$manifold $M$ is called
neutral if it has signature $(n,n)$. A pair $(M,g)$ is called a
pseudo-Riemannian manifold. An anti-K\"{a}hler structure on a manifold $M$
consists of an almost complex structure $J$ and a neutral metric $g$
satisfying the followings:

\textbullet\ algebraic conditions

$(a)$ $J$ is an almost complex structure: $J^{2}=-id.$

$(b)$ The neutral metric $g$ is anti-Hermittian relative to $J$:%
\begin{equation*}
g(JX,JY)=-g(X,Y)
\end{equation*}%
or equivalently
\begin{equation}
g(JX,Y)=g(X,JY),\forall X,Y\in TM.  \label{GC}
\end{equation}

\textbullet\ analytic condition

$(c)$ $J$ is parallel relative to the Levi-Civita connection $\nabla ^{g}$ $%
(\nabla ^{g}J=0)$. This condition is equivalent to ${\Phi }_{J}g=0$, where ${%
\Phi }_{J}$ is the Tachibana operator \cite{IscanSalimov:2009}.

Obviously, by algeraic conditions, the triple $(M,J,g)$ is an almost
anti-Hermitian manifold. Given the anti-Hermitian structure $(J,g)$ on a
manifold $M$, we can immediately recover the other anti-Hermitian metric,
called the twin metric, by the formula:%
\begin{equation*}
G(X,Y)=(g\circ J)(X,Y)=g(JX,Y).
\end{equation*}%
Thus, the triple $(M,J,G)$ is another an almost anti-Hermitian manifold.
Notes that the condition (\ref{GC}) also refers to the purity of $g$
relative to $J$. From now on, by manifold we understand a smooth $2n-$%
manifold and will use the notations $J$, $g$ and $G$ for the almost complex
structure, the pseudo-Riemannian metric and the twin metric, respectively.
In addition, we shall assign the quadruple $(M,J,g,G)$ as almost
anti-Hermitian manifolds.

Our paper aims to study Codazzi pairs on an almost anti-Hermitian manifold $%
(M,J,g,G)$. The structure of the paper is as follows. In Sect. 2, we start
by the $g-$conjugation, $G-$conjugation and $J-$conjugation of arbitrary
linear connections. Then we state the relations among the $(0,4)-$curvature
tensors of these conjugate connections and also show that the set which has $%
g-$conjugation, $G-$conjugation, $J-$conjugation and an identity operation
is a Klein group on the space of linear connections. In Sect. 3, we obtain
some remarkable results under the assumption that $(\nabla ,G)$ or $(\nabla
,J)$ being a Codazzi pair, where $\nabla $ is a linear connection. One of
them is a necessary and sufficient condition under which the almost
anti-Hermitian manifold $(M,J,g,G)$ is an anti-K\"{a}hler relative to a
torsion-free linear connection $\nabla $. Sect. 4 closes our paper with
statistical structures under the assumption that $\nabla $ being $J-$%
invariant relative to a torsion-free linear connection $\nabla $.

\section{Conjugate connections}

In the following let $(M,J,g,G)$ be an almost anti-Hermitian manifold and $%
\nabla $ be a linear connection. We define respectively the conjugate
connections of $\nabla $ relative to $g$ and $G$ as the linear connections
determined by the equations:%
\begin{equation*}
Zg\left( X,Y\right) =g\left( {\nabla }_{Z}X,Y\right) +g\left( X,{{\nabla }%
_{Z}^{\ast }}Y\right)
\end{equation*}%
and

\begin{equation*}
ZG\left( X,Y\right) =G\left( {\nabla }_{Z}X,Y\right) +G\left( X,{{\nabla }%
_{Z}^{\dagger }}Y\right)
\end{equation*}%
for all vector fields $X,Y,Z$ on $M$. We are calling these connections $g-$%
conjugate connection and $G-$conjugate connection, respectively. Note that
both $g-$conjugate connection and $G-$conjugate connection of a linear
connection are involutive: ${\left( {\nabla }^{\ast }\right) }^{\ast
}=\nabla $ and ${\left( {\nabla }^{\dagger }\right) }^{\dagger }=\nabla $.
Conjugate connections are a natural generalization of Levi-Civita
connections from Riemannian manifolds theory. Especially, ${\nabla }^{\ast }$
(or ${\nabla }^{\dagger })$ coincides with $\nabla $ if and only if $\nabla $
is the Levi-Civita connection of $g$ (or $G)$.

Given a linear connection $\nabla $ of $(M,J,g,G)$, the $J-$conjugate
connection of $\nabla $, denoted ${\nabla }^{J}$, is a new linear connection
given by%
\begin{equation*}
{\nabla }^{J}(X,Y)=J^{-1}(\nabla _{X}JY)
\end{equation*}%
for any vector fields $X$ and $Y$ on $M$ \cite{Simon}. Since conjugate
connections arise from affine differential geometry and from geometric
theory of statistical inferences, many studies have been carried out in the
recent years \cite{Amari,Nomizu,Nomizu2}.

Through relationships among the $g-$conjugate connection ${\nabla }^{\ast }$%
, $G-$conjugate connection ${\nabla }^{\dagger }$ and $J-$conjugate
connection ${\nabla }^{J}$ of $\nabla $, we have the following result.

\begin{theorem}
Let $(M,J,g,G)$ be an almost anti-Hermitian manifold. ${\nabla }^{\ast }$, ${%
\nabla }^{\dagger }$ and ${\nabla }^{J}$ denote respectively $g-$%
conjugation, $G-$conjugation and $J-$conjugation of a linear connection $%
\nabla $. Then $(id,\ast ,\dagger ,J)$ acts as the 4-element Klein group on
the space of linear connections:%
\begin{eqnarray*}
i)\text{ }{\left( {\nabla }^{\ast }\right) }^{\ast } &=&{\left( {\nabla }%
^{\dagger }\right) }^{\dagger }={\left( {\nabla }^{J}\right) }^{J}=\nabla ,
\\
ii)\text{ }{\left( {\nabla }^{\dagger }\right) }^{J} &=&{\left( {\nabla }%
^{J}\right) }^{\dagger }={\nabla }^{\ast }, \\
iii)\text{ }{\left( {\nabla }^{\ast }\right) }^{J} &=&{\left( {\nabla }%
^{J}\right) }^{\ast }={\nabla }^{\dagger }, \\
iv)\text{ }{\left( {\nabla }^{\ast }\right) }^{\dagger } &=&{\left( {\nabla }%
^{\dagger }\right) }^{\ast }={\nabla }^{J}.
\end{eqnarray*}
\end{theorem}

\begin{proof}
\textit{i)} The statement is a direct consequence of definitions of
conjugate connections.

\textit{ii)} We compute%
\begin{eqnarray*}
G\left( {{\left( {\nabla }^{\dagger }\right) }_{Z}^{J}}X,Y\right) &=&G\left(
J^{-1}{{\nabla }_{Z}^{\dagger }}\left( JX\right) ,Y\right) \\
&=&G\left( {{\nabla }_{Z}^{\dagger }}\left( JX\right) ,J^{-1}Y\right) \\
&=&ZG\left( JX,J^{-1}Y\right) -G(JX,{\nabla }_{Z}\left( J^{-1}Y\right) ) \\
&=&Zg\left( J^{2}X,\ J^{-1}Y\right) -g(J^{2}X,{\nabla }_{Z}\left(
J^{-1}Y\right) ) \\
&=&-Zg\left( X,J^{-1}Y\right) +g(X,{\nabla }_{Z}\left( J^{-1}Y\right) ) \\
&=&-g\left( {{\nabla }_{Z}^{\ast }}X,J^{-1}Y\right) =G({{\nabla }_{Z}^{\ast }%
}X,Y)
\end{eqnarray*}%
which gives ${{\left( {\nabla }^{\dagger }\right) }^{J}={\nabla }^{\ast }}$.
Similarly%
\begin{eqnarray*}
ZG\left( X,Y\right) &=&G\left( {{\nabla }_{Z}^{J}}X,Y\right) +G\left( X,{{%
\left( {\nabla }^{J}\right) }_{Z}^{\dagger }}Y\right) , \\
Zg\left( JX,Y\right) &=&g\left( {JJ^{-1}\nabla }_{Z}\left( JX\right)
,Y\right) +g\left( JX,{{\left( {\nabla }^{J}\right) }_{Z}^{\dagger }}%
Y\right) , \\
Zg\left( JX,Y\right) &=&g\left( {\nabla }_{Z}\left( JX\right) ,Y\right)
+g\left( JX,{{\left( {\nabla }^{J}\right) }_{Z}^{\dagger }}Y\right) , \\
g(JX,{{\nabla }_{Z}^{\ast }}Y) &=&g\left( JX,{{\left( {\nabla }^{J}\right) }%
_{Z}^{\dagger }}Y\right)
\end{eqnarray*}%
which establishes ${{\left( {\nabla }^{J}\right) }^{\dagger }={\nabla }%
^{\ast }}$. Hence, we get ${{\left( {\nabla }^{\dagger }\right) }^{J}={%
\left( {\nabla }^{J}\right) }^{\dagger }=\nabla }$.

\textit{iii)} On applying the $J-$conjugation to both sides of $ii)$, ${{%
\nabla }^{\dagger }}={(\nabla }^{\ast })^{J}$ and also,%
\begin{eqnarray*}
g\left( JX,{{\left( {\nabla }^{J}\right) }_{Z}^{\ast }}Y\right) &=&\
Zg\left( JX,Y\right) -g\left( {{\nabla }_{Z}^{J}}\left( JX\right) ,Y\right)
\\
&=&\ ZG\left( X,Y\right) -G\left( {J^{-1}{\nabla }_{Z}^{J}}\left( JX\right)
,Y\right) \\
&=&\ ZG\left( X,Y\right) -G\left( {J^{-1}J^{-1}\nabla }_{Z}\left(
J^{2}X\right) ,Y\right) \\
&=&ZG\left( X,Y\right) -G\left( {\nabla }_{Z}X,Y\right) \\
&=&\ G\left( X,{{\nabla }_{Z}^{\dagger }}Y\right) =\ g\left( JX,{{\nabla }%
_{Z}^{\dagger }}Y\right) .
\end{eqnarray*}%
These show that ${{\nabla }^{\dagger }}={(\nabla }^{\ast })^{J}={{\left( {%
\nabla }^{J}\right) }^{\ast }}${.}

\textit{iv) }On applying the $G-$conjugation to both sides of $ii)$, ${%
\nabla }^{J}={\left( {\nabla }^{\ast }\right) }^{\dagger }$ and on applying
the $g-$conjugation to both sides of $iii)$, ${\nabla }^{J}={\left( {\nabla }%
^{\dagger }\right) }^{\ast }$. Thus, the proof completes.
\end{proof}

Recall that the curvature tensor field $R$ of a linear connection $\nabla $
is the tensor field, for all vector fields $X,Y,Z$,%
\begin{equation*}
R(X,Y)Z=\nabla _{X}\nabla _{Y}Z-\nabla _{Y}\nabla _{X}Z-\nabla _{\lbrack
X,Y]}Z.
\end{equation*}%
If $(M,g)$ is a (pseudo-)Riemannian manifold, it is sometimes convenient to
view the curvature tensor field as a $(0,4)-$tensor field by:%
\begin{equation*}
R(X,Y,Z,W)=g(R(X,Y)Z,W)
\end{equation*}%
called the $(0,4)-$curvature tensor field. If we consider the relationship
among the $(0,4)-$curvature tensor fields of $\nabla $, ${\nabla }^{\ast }$
and ${\nabla }^{J}$, we obtain the following.

\begin{theorem}
Let $(M,J,g,G)$ be an almost anti-Hermitian manifold. ${\nabla }^{\ast }$
and ${\nabla }^{J}$ denote respectively $g-$conjugation and $J-$conjugation
of a linear connection $\nabla $ on $M$. The relationship among the $(0,4)-$%
curvature tensor fields $R,R^{\ast }$ and $R^{J}$ of $\nabla $, ${\nabla }%
^{\ast }$ and ${\nabla }^{J}$ is as follow:%
\begin{equation*}
R\left( X,Y,JZ,W\right) =-R^{\ast }\left( X,Y,W,JZ\right) =R^{J}(X,Y,Z,JW)
\end{equation*}%
for \ all vector fields $X,Y,Z,W$ on $M$.
\end{theorem}

\begin{proof}
Since the relation is linear in the arguments $X,$ $Y,W$ and $Z$, it
suffices to prove it only on a basis. Therefore we assume $X,Y,W,Z\in \{%
\frac{\partial }{\partial x^{1}},...,\frac{\partial }{\partial x^{2n}}\}$
and take computational advantage of the following vanishing Lie brackets%
\begin{equation*}
\lbrack X,Y]=[Y,W]=[W,Z]=0.
\end{equation*}%
Then we get
\begin{eqnarray*}
XYG\left( Z,W\right) &=&X\left( Yg\left( JZ,W\right) \right) \\
&=&X(g\left( {\nabla }_{Y}JZ,W\right) )+X\left( g\left( JZ,{{\nabla }%
_{Y}^{\ast }}W\right) \right) \\
&=&g\left( {{\nabla }_{X}\nabla }_{Y}JZ,W\right) +g\left( {\nabla }_{Y}JZ,{{%
\nabla }_{X}^{\ast }}W\right) \\
&&+g\left( {\nabla }_{X}JZ,{{\nabla }_{Y}^{\ast }}W\right) +g\left( JZ,{{{{%
\nabla }_{X}^{\ast }}\nabla }_{Y}^{\ast }}W\right)
\end{eqnarray*}%
and by alternation
\begin{eqnarray*}
YXG\left( Z,W\right) &=&g\left( {{\nabla }_{Y}\nabla }_{X}JZ,W\right)
+g\left( {\nabla }_{X}JZ,{{\nabla }_{Y}^{\ast }}W\right) \\
&&+g\left( {\nabla }_{Y}JZ,{{\nabla }_{X}^{\ast }}W\right) +g\left( JZ,{{{{%
\nabla }_{Y}^{\ast }}\nabla }_{X}^{\ast }}W\right) .
\end{eqnarray*}%
Because of the above relations, we find%
\begin{eqnarray*}
0 &=&\left[ X,Y\right] G\left( Z,W\right) =XYG\left( Z,W\right) -YXG\left(
Z,W\right) \\
0 &=&g\left( {{\nabla }_{X}\nabla }_{Y}JZ-{{\nabla }_{Y}\nabla }%
_{X}JZ,W\right) +g\left( JZ,{{{{\nabla }_{X}^{\ast }}\nabla }_{Y}^{\ast }}W-{%
{{{\nabla }_{Y}^{\ast }}\nabla }_{X}^{\ast }}W\right) \\
0 &=&R\left( X,Y,JZ,W\right) +R^{\ast }\left( X,Y,W,JZ\right)
\end{eqnarray*}%
and similarly%
\begin{eqnarray*}
0 &=&\left[ X,Y\right] G\left( Z,W\right) =XYG\left( Z,W\right) -YXG\left(
Z,W\right) \\
0 &=&G\left( {{J^{-1}\nabla }_{X}}J(J{^{-1}\nabla }_{Y}JZ)-J^{-1}{{\nabla }%
_{Y}}J(J{^{-1}\nabla }_{X}JZ),W\right) \\
&&+G\left( Z,{{{{\nabla }_{X}^{\ast }}\nabla }_{Y}^{\ast }}W-{{{{\nabla }%
_{Y}^{\ast }}\nabla }_{X}^{\ast }}W\right) \\
0 &=&G\left( {{{\nabla }_{X}^{J}\nabla }_{Y}^{J}}Z-{{{\nabla }_{Y}^{J}\nabla
}_{X}^{J}}Z,W\right) \\
&&+G\left( Z,{{{{\nabla }_{X}^{\ast }}\nabla }_{Y}^{\ast }}W-{{{{\nabla }%
_{Y}^{\ast }}\nabla }_{X}^{\ast }}W\right) \\
0 &=&g\left( {{{\nabla }_{X}^{J}\nabla }_{Y}^{J}}Z-{{{\nabla }_{Y}^{J}\nabla
}_{X}^{J}}Z,JW\right) \\
&&+g\left( {{{{\nabla }_{X}^{\ast }}\nabla }_{Y}^{\ast }}W-{{{{\nabla }%
_{Y}^{\ast }}\nabla }_{X}^{\ast }}W,JZ\right) \\
0 &=&R^{J}(X,Y,Z,JW)+R^{\ast }\left( X,Y,W,JZ\right) .
\end{eqnarray*}%
Hence, it follows that $R\left( X,Y,JZ,W\right) =-R^{\ast }\left( X,Y,W,JZ\
\right) =R^{J}(X,Y,Z,JW)$.
\end{proof}

\section{\noindent \noindent \noindent \noindent Codazzi Pairs}

Let $\nabla $ be an arbitrary linear connection on a pseudo-Riemannian
manifold $(M,g)$. Given the pair $(\nabla ,g)$, we construct respectively
the $(0,3)-$tensor fields $F$ and $F^{\ast }$ by%
\begin{equation*}
F(X,Y,Z):=(\nabla _{Z}g)(X,Y)
\end{equation*}%
and%
\begin{equation*}
F^{\ast }(X,Y,Z):=(\nabla _{Z}^{\ast }g)(X,Y),
\end{equation*}%
where $\nabla ^{\ast }$ is $g-$conjugation of $\nabla $. The tensor field $F$
(or $F^{\ast }$) is sometimes referred to as the cubic form associated to
the pair $(\nabla ,g)$ (or $(\nabla ^{\ast },g)$). These tensors are related
via%
\begin{equation*}
F(X,Y,Z)=g(X,(\nabla ^{\ast }-\nabla )_{Z}Y)
\end{equation*}%
so that
\begin{equation*}
F^{\ast }(X,Y,Z):=(\nabla _{Z}^{\ast }g)(X,Y)=-F(X,Y,Z).
\end{equation*}%
Therefore $F(X,Y,Z)=F^{\ast }(X,Y,Z)=0$ if and only if $\nabla ^{\ast
}=\nabla $, that is, $\nabla $ is $g-$self-conjugate \cite{Zhang}.

For an almost complex structure $J$, a pseudo-Riemannian metric $g$ and a
symmetric bilinear form $\rho $ on a manifold $M$, we call $(\nabla ,J)$ and
$(\nabla ,\rho )$, respectively, a Codazzi pair, if their covariant
derivative $(\nabla J)$ and $(\nabla \rho )$, respectively, is (totally)
symmetric in $X,Y,Z$: \cite{Simon}
\begin{equation*}
\left( {\nabla }_{Z}J\right) X=\left( {\nabla }_{X}J\right) Z\text{, }\left(
{\nabla }_{Z}\rho \right) \left( X,Y\right) =\left( {\nabla }_{X}\rho
\right) \left( Z,Y\right) \text{.}
\end{equation*}

\subsection{The Codazzi pair $(\protect\nabla ,G)$}

Let $\nabla $ be a linear connection $\nabla $ on $(M,J,g,G)$. Next we shall
consider the Codazzi pair $(\nabla ,G)$. In here, the $(0,3)-$tensor field $%
F $ is defined by
\begin{equation*}
F(X,Y,Z):=\left( {\nabla }_{Z}G\right) \left( X,\ Y\right) .
\end{equation*}

\begin{proposition}
\label{pr1}Let $\nabla $ be a linear connection on $(M,J,g,G)$. If $(\nabla
,G)$ is a Codazzi pair, then the following statements hold:

$i)$ $F\left( X,Y,Z\right) =\left( {\nabla }_{Z}G\right) \left( X,Y\right) $
is totally symmetric,

$ii)$ $\left( {{\nabla }_{JZ}^{\ast }}G\right) \left( X,Y\right) =\left( {{%
\nabla }_{JX}^{\ast }}G\right) \left( Z,Y\right) ,$

$iii)$ $T^{\nabla }=T^{{\nabla }^{\ast }}$ if and only if $(\nabla ^{\ast
},J)$ is a Codazzi pair,

$iv)$ $T^{\nabla }=T^{{\left( {\nabla }^{\ast }\right) }^{J}}$,

where $\nabla ^{\ast }$ is the $g-$conjugation of $\nabla $ and ${\left( {%
\nabla }^{\ast }\right) }^{J}$ is the $J-$conjugation of $\nabla ^{\ast }$.
\end{proposition}

\begin{proof}
$i)$ Due to symmetry of $G$, $F(X,Y,Z)=\left( {\nabla }_{Z}G\right) \left(
X,Y\right) =\left( {\nabla }_{Z}G\right) \left( Y,X\right) =F(Y,X,Z)$. Also
for $({\nabla },G)$ being a Codazzi pair, $F(X,Y,Z)=\left( {\nabla }%
_{Z}G\right) \left( X,Y\right) =F(X,Y,Z)=\left( {\nabla }_{X}G\right) \left(
Z,Y\right) =F(Z,Y,X)$, that is, $F$ is totally symmetric in all of its
indices.

$ii)$ By virtue of the purity of $g$ relative to $J$, we yield%
\begin{equation*}
\left( {\nabla }_{Z}G\right) \left( X,Y\right) =\left( {\nabla }_{X}G\right)
\left( Z,\ Y\right)
\end{equation*}%
\begin{eqnarray*}
&&Zg\left( JX,Y\right) -g\left( J{\nabla }_{Z}X,Y\right) -g\left( JX,{\nabla
}_{Z}Y\right) \\
&=&Xg\left( JZ,Y\right) -g\left( J{\nabla }_{X}Z,Y\right) -g\left( JZ,{%
\nabla }_{X}Y\right)
\end{eqnarray*}%
\begin{equation*}
g\left( {{\nabla }_{Z}^{\ast }}\left( JX\right) ,Y\right) -g\left( J{\nabla }%
_{Z}X,Y\right) =g\left( {{\nabla }_{X}^{\ast }}\left( JZ\right) ,\ Y\right)
-g\left( J{\nabla }_{X}Z,Y\right)
\end{equation*}%
\begin{eqnarray*}
&&g\left( {{\nabla }_{Z}^{\ast }}\left( JX\right) ,Y\right) -Zg\left(
X,JY\right) +g\left( X,{{\nabla }_{Z}^{\ast }}\left( JY\right) \right) \\
&=&\ g\left( {{\nabla }_{X}^{\ast }}\left( JZ\right) ,Y\right) -Xg\left(
Z,JY\right) +g\left( Z,{{\nabla }_{X}^{\ast }}\left( JY\right) \right)
\end{eqnarray*}%
\begin{eqnarray*}
&&Zg\left( X,JY\right) -g\left( {{\nabla }_{Z}^{\ast }}\left( JX\right)
,Y\right) -g\left( X,{{\nabla }_{Z}^{\ast }}\left( JY\right) \right) \\
&=&Xg\left( Z,JY\right) -g\left( {{\nabla }_{X}^{\ast }}\left( JZ\right) ,\
Y\right) -g\left( Z,{{\nabla }_{X}^{\ast }}\left( JY\right) \right) .
\end{eqnarray*}%
Putting $X=JX,\ Y=JY\ \ $and$\ \ Z=JZ$ in the last relation, we find%
\begin{eqnarray*}
&&JZg\left( JX,J(JY)\right) -g\left( {{\nabla }_{JZ}^{\ast }}\left(
J(JX)\right) ,JY\right) -g\left( JX,{{\nabla }_{JZ}^{\ast }}\left(
J(JY)\right) \right) \\
&=&JXg\left( JZ,J(JY)\right) -g\left( {{\nabla }_{JX}^{\ast }}\left(
J(JZ)\right) ,JY\right) -g\left( JZ,{{\nabla }_{JX}^{\ast }}\left(
J(JY)\right) \right)
\end{eqnarray*}%
\begin{eqnarray*}
&&JZg\left( JX,Y\right) -g\left( {{\nabla }_{JZ}^{\ast }}X,JY\right)
-g\left( JX,{{\nabla }_{JZ}^{\ast }}Y\right) \\
&=&JXg\left( JZ,Y\right) -g\left( {{\nabla }_{JX}^{\ast }}Z,JY\right)
-g\left( JZ,{{\nabla }_{JX}^{\ast }}Y\right)
\end{eqnarray*}%
\begin{eqnarray*}
&&JZG\left( X,Y\right) -G\left( {{\nabla }_{JZ}^{\ast }}X,Y\right) -G\left(
X,{{\nabla }_{JZ}^{\ast }}Y\right) \\
&=&JXG\left( Z,Y\right) -G\left( {{\nabla }_{JX}^{\ast }}Z,Y\right) -G\left(
Z,{{\nabla }_{JX}^{\ast }}Y\right)
\end{eqnarray*}%
\begin{equation*}
\left( {\nabla }_{JZ}^{\ast }G\right) \left( X,Y\right) =\left( {\nabla }%
_{JX}^{\ast }G\right) \left( Z,Y\right) .
\end{equation*}

$iii)$ Let $T^{\nabla }$ and $T^{{\nabla }^{\ast }}$ be respectively the
torsion tensors of $\nabla $ and its $g-$conjugation $\nabla ^{\ast }$. We
calculate%
\begin{equation*}
\left( {\nabla }_{Z}G\right) \left( X,Y\right) =\left( {\nabla }_{X}G\right)
\left( Z,Y\right)
\end{equation*}%
\begin{eqnarray*}
&&Zg\left( JX,Y\right) -g\left( J{\nabla }_{Z}X,Y\right) -g\left( JX,{\nabla
}_{Z}Y\right) \\
&=&Xg\left( JZ,Y\right) -g\left( J{\nabla }_{X}Z,Y\right) -g\left( JZ,{%
\nabla }_{X}Y\right)
\end{eqnarray*}%
\begin{eqnarray*}
&&g\left( {{\nabla }_{Z}^{\ast }}\left( JX\right) ,Y\right) -g\left( J{%
\nabla }_{Z}X,Y\right) \\
&=&g\left( {{\nabla }_{X}^{\ast }}\left( JZ\right) ,Y\right) -g\left( J{%
\nabla }_{X}Z,Y\right)
\end{eqnarray*}%
\begin{eqnarray*}
&&G\left( J^{-1}{{\nabla }_{Z}^{\ast }}\left( JX\right) ,Y\right) -G\left( {%
\nabla }_{Z}X,Y\right) \\
&=&G\left( {{J^{-1}\nabla }_{X}^{\ast }}\left( JZ\right) ,Y\right) -G\left( {%
\nabla }_{X}Z,Y\right)
\end{eqnarray*}%
\begin{equation}
G\left( J^{-1}\left\{ {{\nabla }_{Z}^{\ast }}\left( JX\right) -{{\nabla }%
_{X}^{\ast }}\left( JZ\right) \right\} ,Y\right) =G\left( {\nabla }_{Z}X-{%
\nabla }_{X}Z,Y\right)  \label{GC1}
\end{equation}%
from which we get%
\begin{equation*}
J^{-1}\left\{ {{\nabla }_{Z}^{\ast }}\left( JX\right) -{{\nabla }_{X}^{\ast }%
}\left( JZ\right) \right\} ={\nabla }_{Z}X-{\nabla }_{X}Z
\end{equation*}%
\begin{equation*}
J^{-1}\left\{ \left( {{\nabla }_{Z}^{\ast }}J\right) X+J{{\nabla }_{Z}^{\ast
}}X-\left( {{\nabla }_{X}^{\ast }}J\right) Z-J{{\nabla }_{X}^{\ast }}%
Z\right\} ={\nabla }_{Z}X-{\nabla }_{X}Z
\end{equation*}%
\begin{eqnarray*}
&&J^{-1}\left\{ \left( {{\nabla }_{Z}^{\ast }}J\right) X-\left( {{\nabla }%
_{X}^{\ast }}J\right) Z\right\} +\left( {{\nabla }_{Z}^{\ast }}X-{{\nabla }%
_{X}^{\ast }}Z-\left[ Z,X\right] \right) \\
&=&{\nabla }_{Z}X-{\nabla }_{X}Z-\left[ Z,X\right]
\end{eqnarray*}%
\begin{equation*}
J^{-1}\left\{ \left( {{\nabla }_{Z}^{\ast }}J\right) X-\left( {{\nabla }%
_{X}^{\ast }}J\right) Z\right\} +T^{{\nabla }^{\ast }}\left( Z,X\right)
=T^{\nabla }(Z,X).
\end{equation*}%
This means that $T^{{\nabla }^{\ast }}\left( Z,X\right) =T^{\nabla }(Z,X)$
if and only if $\left( {{\nabla }_{Z}^{\ast }}J\right) X=\left( {{\nabla }%
_{X}^{\ast }}J\right) Z$.

$iv)$ From (\ref{GC1}), we can write%
\begin{equation*}
G\left( {{\left( {\nabla }^{\ast }\right) }_{Z}^{J}}X-{{\left( {\nabla }%
^{\ast }\right) }_{X}^{J}}Z,Y\right) =\ G\left( {\nabla }_{Z}X-{\nabla }%
_{X}Z,Y\right)
\end{equation*}%
\begin{equation*}
G(T^{{\left( {\nabla }^{\ast }\right) }^{J}}(Z,X),Y)=G(T^{\nabla }(Z,X),Y)
\end{equation*}%
\begin{equation*}
T^{{\left( {\nabla }^{\ast }\right) }^{J}}(Z,X)=T^{\nabla }(Z,X).
\end{equation*}
\end{proof}

Now we shall state the following proposition without proof, because its
proof is similar to the proof of Proposition 2.10 in \cite{Zhang}.

\begin{proposition}
\noindent \label{pr2}Let $\nabla $ be a linear connection on $(M,J,g,G)$.
Then the following statements are equivalent:

$i)$ $(\nabla ,G)$ is a Codazzi pair

$ii)$ $({\nabla }^{\dagger },G)$ is a Codazzi pair,

$iii)$ $F^{\dagger }\left( X,Y,Z\right) =\left( {\nabla }_{Z}^{\dagger
}G\right) \left( X,Y\right) $ is totally symmetric,

$iv)$ $T^{\nabla }=T^{{\nabla }^{\dagger }}.$
\end{proposition}

\noindent\ As a corollary to Proposition \ref{pr1} and \ref{pr2}, we obtain
the following conclusion.

\begin{corollary}
Let $(M,J,g,G)$ be an almost anti-Hermitian manifold. ${\nabla }^{\ast }$
and ${\nabla }^{\dagger }$ denote respectively $g-$conjugation and $G-$%
conjugation of a linear connection $\nabla $ on $M$. If $(\nabla ,G)$ and ${%
\left( {\nabla }^{\ast },J\right) }$ are Codazzi pairs, then $T^{\nabla }=T^{%
{\nabla }^{\ast }}=T^{{\nabla }^{\dagger }}.$
\end{corollary}

\noindent

\subsection{The Codazzi pair $(\protect\nabla ,J)$}

\begin{proposition}
Let $\nabla $ be a linear connection on $(M,J,g,G)$. ${\nabla }^{\dagger }$
denote $G-$conjugation of $\nabla $ on $M$. Under the assumption that $%
(\nabla ,G)$ being a Codazzi pair, $({\nabla }^{\dagger },J)$ is a Codazzi
pair if and only if $({\nabla ,g)}$ is so.
\end{proposition}

\begin{proof}
Using the definition of $G-$conjugation and $T^{\nabla }=T^{{\nabla }%
^{\dagger }}$, we find%
\begin{equation*}
G\left( ({\nabla }_{Z}^{\dagger }J)X-({\nabla }_{X}^{\dagger }J)Z,Y\right)
=G({\nabla }_{Z}^{\dagger }JX-{{J}\nabla }_{Z}^{\dagger }X,Y)-G({\nabla }%
_{X}^{\dagger }JZ-{{J}\nabla }_{X}^{\dagger }Z,Y)
\end{equation*}%
\begin{eqnarray*}
&=&ZG\left( JX,Y\right) -G\left( JX,{\nabla }_{Z}Y\right) -G\left( {{J}%
\nabla }_{Z}^{\dagger }X,Y\right) -XG\left( JZ,Y\right) \\
&&+G\left( JZ,{\nabla }_{X}Y\right) +G({{J}\nabla }_{X}^{\dagger }Z,Y)
\end{eqnarray*}%
\begin{eqnarray*}
&=&ZG\left( JX,Y\right) -G\left( JX,{\nabla }_{Z}Y\right) -XG\left(
JZ,Y\right) +G\left( JZ,{\nabla }_{X}Y\right) \\
&&+G\left( {{J(}\nabla }_{X}^{\dagger }Z-{\nabla }_{Z}^{\dagger }X-\left[ Z,X%
\right] )+J\left[ Z,X\right] ,Y\right)
\end{eqnarray*}%
\begin{eqnarray*}
&=&ZG\left( JX,Y\right) -G\left( JX,{\nabla }_{Z}Y\right) -XG\left(
JZ,Y\right) +G\left( JZ,{\nabla }_{X}Y\right) \\
&&+G\left( {{J(}\nabla }_{X}Z-{\nabla }_{Z}X-\left[ Z,X\right] )+J\left[ Z,X%
\right] ,Y\right)
\end{eqnarray*}%
\begin{eqnarray*}
&=&-Zg\left( X,Y\right) +g\left( X,{\nabla }_{Z}Y\right) +Xg\left( Z,Y\right)
\\
&&-g\left( Z,{\nabla }_{X}Y\right) +g\left( {\nabla }_{Z}X,Y\right) -g({%
\nabla }_{X}Z,Y) \\
&=&({\nabla }_{Z}g)(X,Y)-({\nabla }_{X}g)(Z,Y).
\end{eqnarray*}
\end{proof}

Now we consider the ${\Phi }-$operator (or Tachibana operator \cite%
{Tachibana}) applied to the anti-Hermitian metric $g$:%
\begin{equation}
({\Phi }_{J}g)(X,Y,Z)=(L_{JX}g-L_{X}(g\circ J))(Y,Z).  \label{GC3}
\end{equation}%
Because of the fact that the twin metric $G$ on an almost anti-Hermitian
manifold $(M,J,g)$ is an anti-Hermitian metric, we can apply the ${\Phi }-$%
operator to the twin metric $G$: \cite{Salimov3}%
\begin{eqnarray}
({\Phi }_{J}G)(X,Y,Z) &=&(L_{JX}G-L_{X}(G\circ J))(Y,Z)  \label{GC4} \\
&=&({\Phi }_{J}g)(X,JY,Z)+g(N_{J}(X,Y),Z).  \notag
\end{eqnarray}

\begin{proposition}
\label{pr3}Let $\nabla $ be a torsion-free linear connection on $(M,J,g,G)$.
If $({\nabla },J)$ is a Codazzi pair, then%
\begin{equation*}
({\Phi }_{J}G)(X,Y,Z)=({\Phi }_{J}g)(X,JY,Z)=\left( {\nabla }_{JX}G\right)
\left( Y,Z\right) -\left( {\nabla }_{X}g\right) \left( JY,JZ\right) .
\end{equation*}
\end{proposition}

\begin{proof}
Using ${\nabla }_{X}Z-{\nabla }_{Z}X=\left[ Z,X\right] $, from (\ref{GC3})
we get%
\begin{equation*}
\left( {\Phi }_{J}g\right) \left( X,JY,Z\right) =(L_{JX}g-\left(
L_{X}goJ\right) \left( JY,Z\right)
\end{equation*}%
\begin{equation*}
=\left( L_{JX}g\right) \left( JY,Z\right) -\left( L_{X}goJ\right) \left(
JY,Z\right)
\end{equation*}%
\begin{eqnarray*}
&=&JXg\left( JY,Z\right) -g\left( L_{JX}JY,Z\right) -g\left( {JY,L}%
_{JX}Z\right) -XgoJ\left( JY,Z\right) \\
&&+goJ\left( L_{X}JY,Z\right) +goJ\left( JY,L_{X}Z\right)
\end{eqnarray*}%
\begin{eqnarray*}
&=&JXg\left( JY,Z\right) -g\left( \left[ JX,JY\right] ,Z\right) -g\left( JY,%
\left[ JX,Z\right] \right) -XgoJ\left( JY,Z\right) \\
&&+goJ\left( \left[ X,JY\right] ,Z\right) +goJ\left( JY,\left[ X,Z\right]
\right)
\end{eqnarray*}%
\begin{eqnarray*}
&=&JXg\left( JY,Z\right) -g\left( {\nabla }_{JX}JY-{\nabla }_{JY}JX,Z\right)
-g\left( JY,{\nabla }_{JX}Z-{\nabla }_{Z}JX\right) \\
&&-XgoJ\left( JY,Z\right) +goJ\left( {\nabla }_{X}JY-{\nabla }%
_{JY}X,Z\right) +goJ\left( JY,{\nabla }_{X}Z-{\nabla }_{Z}X\right)
\end{eqnarray*}%
\begin{eqnarray*}
&=&JXg\left( JY,Z\right) -g\left( \left( {\nabla }_{JX}J\right) Y+J{\nabla }%
_{JX}Y-\left( {\nabla }_{JY}J\right) X-J{\nabla }_{JY}X,Z\right) \\
&&-g\left( JY,{\nabla }_{JX}Z-\left( {\nabla }_{Z}J\right) X-J{\nabla }%
_{Z}X\right) -Xg\left( JY,JZ\right) \\
&&+g\left( \left( {\nabla }_{X}J\right) Y+J{\nabla }_{X}Y-{\nabla }%
_{JY}X,JZ\right) +g\left( JY,J{\nabla }_{X}Z-J{\nabla }_{Z}X\right)
\end{eqnarray*}%
\begin{eqnarray*}
&=&JXg\left( JY,Z\right) -g\left( \left( {\nabla }_{JX}J\right) Y,Z\right)
-g\left( J{\nabla }_{JX}Y,Z\right) \\
&&+g\left( \left( {\nabla }_{JY}J\right) X,Z\right) +g\left( J{\nabla }%
_{JY}X,Z\right) -g\left( JY,{\nabla }_{JX}Z\right) +g\left( JY,\left( {%
\nabla }_{Z}J\right) X\right) \\
&&+g\left( JY,J{\nabla }_{Z}X\right) -Xg\left( JY,JZ\right) +g\left( \left( {%
\nabla }_{X}J\right) Y,JZ\right) +g\left( J{\nabla }_{X}Y,JZ\right) \\
&&-g\left( {\nabla }_{JY}X,JZ\right) +g\left( JY,J{\nabla }_{X}Z\right)
-g\left( JY,J{\nabla }_{Z}X\right) .
\end{eqnarray*}%
\bigskip \noindent By virtue of the purity of $g$ relative to $J$, $({\nabla
}_{Z}J)X=({\nabla }_{X}J)Z$, the last relation reduces to
\begin{eqnarray*}
&=&JXg\left( JY,Z\right) -g\left( \left( {\nabla }_{JX}J\right) Y,Z\right)
-g\left( J{\nabla }_{JX}Y,Z\right) +g\left( \left( {\nabla }_{JY}J\right)
X,Z\right) \\
&&+g\left( J{\nabla }_{JY}X,Z\right) -g\left( JY,{\nabla }_{JX}Z\right)
+g\left( JY,\left( {\nabla }_{Z}J\right) X\right) \\
&&+g\left( JY,J{\nabla }_{Z}X\right) -Xg\left( JY,JZ\right) +g\left( \left( {%
\nabla }_{X}J\right) Y,JZ\right) \\
&&+g\left( J{\nabla }_{X}Y,JZ\right) -g\left( J{\nabla }_{JY}X,Z\right)
+g\left( JY,J{\nabla }_{X}Z\right) -g\left( JY,J{\nabla }_{Z}X\right)
\end{eqnarray*}%
\begin{eqnarray*}
&=&JXg\left( JY,Z\right) -g\left( J{\nabla }_{JX}Y,Z\right) -g\left( JY,{%
\nabla }_{JX}Z\right) +g\left( JY,\left( {\nabla }_{Z}J\right) X\right) \\
&&-Xg\left( JY,JZ\right) +g\left( \left( {\nabla }_{X}J\right) Y,JZ\right)
+g\left( J{\nabla }_{X}Y,JZ\right) +g\left( JY,J{\nabla }_{X}Z\right)
\end{eqnarray*}%
\begin{eqnarray*}
&=&JXG\left( Y,Z\right) -G\left( {\nabla }_{JX}Y,Z\right) -G\left( Y,{\nabla
}_{JX}Z\right) -Xg\left( JY,JZ\right) \\
&&+g\left( {\nabla }_{X}JY,JZ\right) +g\left( JY,{\nabla }_{Z}JX\right)
\end{eqnarray*}%
\begin{equation}
=\left( {\nabla }_{JX}G\right) \left( Y,Z\right) -\left( {\nabla }%
_{X}g\right) \left( JY,JZ\right) .  \label{GC5}
\end{equation}

Relative to the torsion-free connection $\nabla $, the Nijenhuis tensor has
the following form:%
\begin{equation*}
N_{J}\left( X,Y\right) =-J\{(\nabla _{JY}J)JX-(\nabla _{JX}J)JY\}+J\{(\nabla
_{Y}J)X-(\nabla _{X}J)Y\}
\end{equation*}%
From here, it is easy to $N_{J}\left( X,Y\right) =0$ because $({\nabla },J)$
is a Codazzi pair. Hence, taking account of (\ref{GC4}) and (\ref{GC5}) we
have%
\begin{equation*}
\left( {\Phi }_{J}G\right) \left( X,Y,Z\right) =\left( {\Phi }_{J}g\right)
\left( X,JY,Z\right) =\left( {\nabla }_{JX}G\right) \left( Y,Z\right)
-\left( {\nabla }_{X}g\right) \left( JY,JZ\right) .
\end{equation*}
\end{proof}

As is well known, the anti-K\"{a}hler condition ($\nabla ^{g}J=0$) is
equivalent to $%
\mathbb{C}
$-holomorphicity (analyticity) of the anti-Hermitian metric $g$, that is, ${%
\Phi }_{J}g=0$. If the anti-Hermitian metric $g$ is $%
\mathbb{C}
-$holomorphic, then the triple $(M,J,g)$ is an anti-K\"{a}hler manifold \cite%
{IscanSalimov:2009}.

\begin{theorem}
Let $\nabla $ be a torsion-free linear connection on $(M,J,g,G)$. Under the
assumption that $({\nabla },J)$ being a Codazzi pair, $(M,J,g,G)$ is an
anti-K\"{a}hler manifold if and only if the following condition is fulfilled:%
\begin{equation*}
\left( {\nabla }_{JX}G\right) \left( Y,Z\right) =\left( {\nabla }%
_{X}g\right) \left( JY,JZ\right) .
\end{equation*}
\end{theorem}

\begin{proof}
The statement is a direct consequence of Proposition \ref{pr3}.
\end{proof}

\section{\noindent $J-$invariant Linear Connections}

Given arbitrary linear connection $\nabla $ on an almost complex manifold $%
(M,J)$, if the following condition is satisfied:%
\begin{equation*}
{\nabla }_{X}JY={J\nabla }_{X}Y
\end{equation*}%
for any vector fields $X,Y$ on $M$, then $\nabla $ is called a $J-$invariant
linear connection on $M$.

\begin{proposition}
\label{pr4} Let $\nabla $ be a linear connection on $(M,J,g,G)$. ${\nabla }%
^{\ast }$ and ${\nabla }^{\dagger }$ denote respectively $g-$conjugation and
$G-$conjugation of $\nabla $ on $M$. Then

$i)$ $\nabla $ is $J-$invariant if and only if ${\nabla }^{\ast }$ is so.

$ii)$ $\nabla $ is $J-$invariant if and only if ${\nabla }^{\dagger }$ is so.
\end{proposition}

\begin{proof}
$i)$ Using the definition of $g-$conjugation and the purity of $g$ relative
to $J$, we have
\begin{equation*}
G\left( {{\nabla }_{X}^{\ast }}JY-{J{\nabla }_{X}^{\ast }}Y,Z\right)
=g\left( {{\nabla }_{X}^{\ast }}JY,JZ\right) -{g(J{\nabla }_{X}^{\ast }}Y,JZ)
\end{equation*}%
\begin{equation*}
=-Xg\left( Y,Z\right) -g\left( JY,{\nabla }_{X}JZ\right) +Xg\left(
Y,Z\right) -g\left( Y,{\nabla }_{X}Z\right)
\end{equation*}%
\begin{equation*}
=-g\left( JY,{\nabla }_{X}\ JZ\right) +g\left( JY,J{\nabla }_{X}Z\right)
=-G\left( Y,{\nabla }_{X}JZ\right) +G\left( Y,J{\nabla }_{X}Z\right) .
\end{equation*}%
Hence, ${{\nabla }_{X}^{\ast }}JY={J{\nabla }_{X}^{\ast }}Y$ if and only if $%
{\nabla }_{X}\ JZ=J{{\nabla }_{X}\ Z}$.

$ii)$ Similarly, we get\noindent
\begin{equation*}
G\left( {{\nabla }_{X}^{\dagger }}JY-J{{\nabla }_{X}^{\dagger }}Y,Z\right)
=G\left( {{\nabla }_{X}^{\dagger }}JY,Z\right) -G(J{{\nabla }_{X}^{\dagger }}%
Y,Z)
\end{equation*}%
\begin{equation*}
=XG\left( JY,Z\right) -G\left( JY,{\nabla }_{X}Z\right) -XG\left(
Y,JZ\right) +G\left( Y,{\nabla }_{X}\ JZ\right)
\end{equation*}%
\begin{equation*}
=G\left( Y,{\nabla }_{X}JZ\right) -G\left( JY,{\nabla }_{X}Z\right) =G\left(
{\nabla }_{X}\ JZ-{J\nabla }_{X}Z,Y\right)
\end{equation*}%
which gives the result.
\end{proof}

\begin{proposition}
\label{pr5}Let $\nabla $ be a $J-$invariant linear connection on $(M,J,g,G)$%
. ${\nabla }^{\ast }$ and ${\nabla }^{\dagger }$ denote respectively $g-$%
conjugation and $G-$conjugation of $\nabla $ on $M$. The following
statements hold:

$i)$ ${\nabla }^{\dagger }$ coincides with ${\nabla }^{\ast }$,

$ii)$ $(\nabla ,G)$ is a Codazzi pair if and only if $(\nabla ,g)$ is so.
\end{proposition}

\begin{proof}
$i)$ By the definition of $g-$conjugation, $G-$conjugation and $J-$%
invariance, we have \textbf{\ }
\begin{equation*}
ZG\left( X,Y\right) =G\left( {\nabla }_{Z}X,Y\right) +G\left( X,{\nabla }%
_{Z}^{\dagger }Y\right)
\end{equation*}%
\begin{equation*}
Zg\left( JX,Y\right) =g\left( J{\nabla }_{Z}X,Y\right) +g\left( JX,{\nabla }%
_{Z}^{\dagger }Y\right)
\end{equation*}%
\begin{equation*}
Zg\left( JX,Y\right) -g\left( J{\nabla }_{Z}X,Y\right) =g\left( JX,{\nabla }%
_{Z}^{\dagger }Y\right)
\end{equation*}%
\begin{equation*}
Zg\left( JX,Y\right) -g\left( {\nabla }_{Z}JX,Y\right) =g\left( JX,{\nabla }%
_{Z}^{\dagger }Y\right)
\end{equation*}%
\begin{equation*}
g\left( JX,{{\nabla }_{Z}^{\ast }}Y\right) =g\left( JX,{\nabla }%
_{Z}^{\dagger }Y\right) \iff {{\nabla }^{\ast }}={\nabla }^{\dagger }
\end{equation*}

$ii)$ Using the purity of $g$ relative to $J$, we get
\begin{equation*}
\left( {\nabla }_{Z}G\right) \left( X,Y\right) =\left( {\nabla }_{X}G\right)
\left( Z,Y\right)
\end{equation*}%
\begin{equation*}
Zg\left( JX,Y\right) -g\left( J{\nabla }_{Z}X,Y\right) -g\left( JX,{\nabla }%
_{Z}Y\right) =Xg\left( JZ,Y\right) -g\left( J{\nabla }_{X}Z,Y\right)
-g\left( JZ,{\nabla }_{X}Y\right)
\end{equation*}%
\begin{equation*}
Zg\left( X,JY\right) -g\left( {\nabla }_{Z}X,JY\right) -g\left( X,{J\nabla }%
_{Z}Y\right) =Xg\left( Z,JY\right) -g\left( {\nabla }_{X}Z,JY\right)
-g\left( Z,{J\nabla }_{X}Y\right)
\end{equation*}%
\begin{equation*}
Zg\left( X,JY\right) -g\left( {\nabla }_{Z}X,JY\right) -g\left( X,{\nabla }%
_{Z}JY\right) =Xg\left( Z,JY\right) -g\left( {\nabla }_{X}Z,JY\right)
-g\left( Z,{\nabla }_{X}JY\right)
\end{equation*}%
\begin{equation*}
\left( {\nabla }_{Z}g\right) \left( X,JY\right) =\left( {\nabla }%
_{X}g\right) \left( Z,JY\right) .
\end{equation*}
\end{proof}

For the moment, we consider a torsion-free linear connection $\nabla $ on a
pseudo-Riemannian manifold $(M,g)$. In the case, if $(\nabla ,g)$ is a
Codazzi pair which characterizes what is known to information geometers as
statistical structures, then the manifold $M$ together with a statistical
structure $(\nabla ,g)$ is called a statistical manifold. The notion of
statistical manifold was originally introduced by Lauritzen \cite{Lauritzen}%
. Statistical manifolds are widely studied in affine differential geometry
\cite{Lauritzen,Nomizu} and plays a central role in information geometry.

\begin{theorem}
Let $\nabla $ be a $J-$invariant torsion-free linear connection on $%
(M,J,g,G) $. ${\nabla }^{\dagger }$ and ${\nabla }^{\ast }$ denote
respectively the $G- $conjugation and $g-$conjugation of $\nabla $ on $M$.
If $({\nabla },G)$ is a statistical structure, then the following statements
hold:\noindent

$i)$ $({\nabla }^{\dagger },g)$ is a statistical structure,

$ii)$ $\left( \nabla ,g\right) $ is a statistical structure,

$iii)$ $({\nabla }^{\ast },g)$ is a statistical structure.

Conversely, if any one of the statements $i)-iii)$ is satisfied, then $({%
\nabla },G)$ is a statistical structure.
\end{theorem}

\begin{proof}
The result comes directly from Proposition \ref{pr4} and \ref{pr5}.
\end{proof}

\begin{theorem}
\noindent \noindent Let $\nabla $ be a $J-$invariant torsion-free linear
connection on $(M,J,g,G)$. ${\nabla }^{\ast }$ denote the $g-$conjugation of
$\nabla $ on $M$. $({\nabla },G)$ is a statistical structure if and only if $%
({\nabla }^{\ast },G)$ is so.
\end{theorem}

\begin{proof}
The result immediately follows from Proposition \ref{pr1}, using the
condition of $\nabla $ being $J-$invariant.
\end{proof}

\end{document}